\documentclass[11pt]{amsart}

\usepackage{amsmath}
\usepackage{amsfonts}
\usepackage{amssymb}
\usepackage{amscd}
\usepackage{amsthm}
\usepackage{filecontents}
\usepackage{framed}
\usepackage{fullpage}
\usepackage{latexsym}
\usepackage{enumerate}
\usepackage{hyperref}
\hypersetup{colorlinks,linkcolor={red},citecolor={blue},urlcolor={blue}}

\evensidemargin\oddsidemargin

\newcommand{\C}{\mathbb C}

\newcommand{\Q}{\mathbb Q}
\newcommand{\R}{\mathbb R}
\newcommand{\Z}{\mathbb Z}

\newcommand{\mult}{\operatorname{mult}}
\newcommand{\sign}{\operatorname{sign}}
\newcommand{\odd}{\operatorname{oddity}}
\newcommand{\ord}{\operatorname{ord}}

\newcommand{\im}{\operatorname{Im}}
\newcommand{\rk}{\operatorname{rk}}

\newcommand{\abs}[1]{\lvert#1\rvert}

\newcommand{\SL}{\operatorname{SL}}
\newcommand{\Sp}{\operatorname{Sp}}

\newcommand{\Co}{\operatorname{Co}}

\newcommand{\cD}{\mathcal D}
\newcommand{\PP}{\mathbb P}
\newcommand{\Orth}{\mathop{\null\mathrm {O}}\nolimits}

\newcommand{\HH}{\mathbb H}
\newcommand{\cH}{\mathcal H}

\newcommand{\Grit}{\operatorname{G}}

\newcommand{\latt}[1]{{\langle{#1}\rangle}}

\newcommand{\II}{I\! I}
\newcommand{\legendre}[2]{\left(\frac{#1}{#2}\right)}
\numberwithin{equation}{section}

\newtheorem{thm}{Theorem}[section]
\newtheorem{prp}[thm]{Proposition}
\newtheorem{cor}[thm]{Corollary}
\newtheorem{lem}[thm]{Lemma}
\newtheorem*{thm*}{Theorem}
\newtheorem*{cor*}{Corollary}
\newtheorem*{conjecture*}{Conjecture}

\theoremstyle{definition}
\newtheorem{definition}[thm]{Definition}

\newtheorem*{definition*}{Definition}

\theoremstyle{remark}
\newtheorem{remark}[thm]{Remark}

\begin{document}

\title[Theta blocks related to root systems]{Theta blocks related to root systems}

\author{Moritz Dittmann}

\address{Fachbereich Mathematik, Technische Universit\"{a}t Darmstadt, Darmstadt, Germany}

\email{mdittmann@mathematik.tu-darmstadt.de}

\author{Haowu Wang}

\address{Center for Geometry and Physics, Institute for Basic Science (IBS), Pohang 37673, Korea}

\email{haowu.wangmath@gmail.com}


\date{\today}


\begin{abstract}
Gritsenko, Skoruppa and Zagier associated to a root system $R$ a theta block $\vartheta_R$, which is a Jacobi form of lattice index. We classify the theta blocks $\vartheta_R$ of $q$-order $1$ and show that their Gritsenko lift is a strongly-reflective Borcherds product of singular weight, which is related to Conway's group $\Co_0$. As a corollary we obtain a proof of the theta block conjecture by Gritsenko, Poor and Yuen for the pure theta blocks obtained as specializations of the functions $\vartheta_R$.
\end{abstract}

\maketitle

\section{Introduction}
Eichler and Zagier introduced the theory of Jacobi forms in their monograph \cite{EZ85}. 
Let $k$ and $m$ be non-negative half-integers and $\chi$ a character (or multiplier system) of $\SL_2(\Z)$. A holomorphic Jacobi form of weight $k$, character $\chi$  and index $m$ is a holomorphic function $\varphi\colon \HH\times \C \to \C$ which satisfies 
\[
\varphi\Big(\frac{a\tau+b}{c\tau+d},\frac{z}{c\tau+d}\Big)=\chi\left(\begin{pmatrix}
a & b \\ c & d
\end{pmatrix}\right)\sqrt{c\tau+d}^{2k}e^{2\pi i\frac{mc z^2}{c\tau+d}}\varphi(\tau,z)
\] and
\[
\varphi(\tau,z+\lambda\tau+\mu)=(-1)^{2m(\lambda+\mu)}e^{-2\pi im(\lambda^2\tau+2\lambda z)}\varphi(\tau,z)
\]
for all $\tau\in \HH, z\in \C, \left(\begin{smallmatrix}
a & b \\ c & d
\end{smallmatrix} \right)\in \SL_2(\Z)$ and $\lambda,\mu\in \Z$ and which has a Fourier expansion of the form
\[
\varphi(\tau,z)=\sum_{\substack{n\in \Q \\n\geq 0}}\sum_{\substack{r\in \Z \\ r^2\leq 4mn}}c(n,r)q^ne^{2\pi irz}, \quad q^n=e^{2\pi in \tau}.
\]
Examples of holomorphic Jacobi forms of small weight and index are the Dedekind eta function
\[
\eta(\tau)=q^{1/24}\prod_{n=1}^\infty(1-q^n)
\]
of weight $1/2$ and index $0$ with a multiplier system which we denote by $\nu_\eta$ (note that Jacobi forms of index $0$ do not depend on $z$ and their definition reduces to that of a classical modular form) and the Jacobi theta function of weight and index $1/2$ and multiplier system $\nu_\eta^3$, given by
\[
\vartheta(\tau,z)=\sum_{n=-\infty}^{\infty}\legendre{-4}{n}q^{n^2/8}e^{\pi inz},
\]
or by the triple product identity
\[
\vartheta(\tau,z)=q^{1/8}e^{\pi i z}\prod_{n=1}^{\infty}(1-q^n)(1-q^ne^{2\pi i z})(1-q^{n-1}e^{-2\pi i z}).
\]
For a non-zero integer $a$ we denote by $\vartheta_a$ the function
\[
\vartheta_a(\tau,z)=\vartheta(\tau,az).
\]
This is a Jacobi form of weight $1/2$ and index $a^2/2$. More generally, to a function  $f\colon \Z_{\geq 0}\to\Z$ with finite support, we associate a theta block
\[
\Theta_f(\tau,z)=\eta^{f(0)}(\tau)\prod_{a=1}^\infty (\vartheta_a(\tau,z)/\eta(\tau))^{f(a)},
\]
which is a meromorphic Jacobi form. If the image of $f$ is contained in the non-negative integers, then $\Theta_f$ is called a pure theta block.   For more details on the theory of theta blocks, we refer the reader to \cite{GSZ}.

Jacobi forms can be used to construct paramodular forms. These are Siegel modular forms of degree two with respect to the paramodular group
\[
\Gamma_N=\begin{pmatrix}
\ast & N\ast & \ast & \ast \\ \ast & \ast & \ast & \ast/N \\ \ast & N\ast & \ast & \ast \\ N\ast & N\ast & N\ast &\ast
\end{pmatrix}\cap \Sp_2(\Q), \quad \text{ all } \ast\in \Z
\] of some level $N$. One method to construct paramodular forms is the Gritsenko lift, which sends a holomorphic Jacobi form $\varphi$ to a paramodular form $G(\varphi)$ of the same weight. Another method associates to a nearly holomorphic Jacobi form $\psi$ of weight $0$ with integral singular Fourier coefficients a meromorphic paramodular form $B(\psi)$. This method is essentially the multiplicative Borcherds lift. In \cite{GPY}, Gritsenko, Poor and Yuen investigated paramodular forms which are simultaneously Borcherds products and Gritsenko lifts. From the shapes of the arising paramodular forms, one sees that if $G(\varphi)$ is a Borcherds product, then $\varphi$ must be a theta block with vanishing order one in $q$.

In \cite{GPY}, the following conjecture, which gives a sufficient condition for $G(\varphi)$ being a Borcherds product, was formulated.

\begin{conjecture*}[Theta Block Conjecture]
    Let the pure theta block $\Theta_f$ be a holomorphic Jacobi form of weight $k$ and index $m$ with vanishing order $1$ in $q$, where $k,m\in \Z_{>0}$. We define the nearly holomorphic Jacobi form $\Psi_f=-(\Theta_f|T_-(2))/\Theta_f$ of weight $0$ and index $m$, where $T_-(2)$ is the index raising Hecke operator. Then
    \[
    G(\Theta_f)=B(\Psi_f).
    \]
\end{conjecture*}

In this paper we prove a higher-dimensional analogue of the theta block conjecture for certain Jacobi forms $\vartheta_R$ in many variables. More precisely, to a root system $R$ we can attach a holomorphic Jacobi form $\vartheta_R$ of weight $k=\rk(R)/2$ and lattice index $\underline{R}$ (see Theorem \ref{thm:RootThetaFunction}). The Borcherds and Gritsenko lifts of a classical Jacobi form are special cases of more general Borcherds and Gritsenko lifts for Jacobi forms of lattice index. Their images are modular forms for orthogonal groups of signature $(2,n)$ (in the case of a classical Jacobi form, $n=3$ and paramodular forms arise because they can be realized as modular forms for orthogonal groups of signature $(2,3)$). Our main result is the following theorem.

\begin{thm*}[Theorem \ref{thm:mainThm}]
    Let $R$ be a root system such that $\vartheta_R$ has vanishing order $1$ in $q$. Then 
    \[
    G(\vartheta_R)=B\Big(-\frac{\vartheta_R | T_{-}(2)}{\vartheta_R}\Big).
    \]
\end{thm*}
In particular, $G(\vartheta_R)$ is a Borcherds product. It turns out that this Borcherds product already appears in the work of Scheithauer \cite{Scheithauer-Class, Scheithauer-Moonshine} and its expansion at a level $1$ cusp is a twisted denominator identity of the fake monster algebra corresponding to an element $g$ in Conway's group $\Co_0$.

The theorem is proved by showing that the divisor of the right hand side is contained in the divisor of $G(\vartheta_R)$ for all possible choices of $R$. There are eight such root systems $R$. We remark that for $R=8A_1, 3A_2$ and $A_4$ this proof can already be found in the literature (see \cite[Theorem 5.2]{Gritsenko-Reflective} for $8A_1$, \cite[Theorem 5.6]{Gritsenko-Reflective} for $3A_2$ and \cite[Theorem 3.9]{GW} for $A_4$). 

The specialization $\Theta_{x}$ of $\vartheta_R$ at a non-zero vector $x\in \underline{R}$ is defined by $\Theta_x(\tau,z)=\vartheta_R(\tau,xz)$. We only consider vectors $x\in \underline{R}$ such that $\Theta_x$ is not identically zero and has integral index. Then $\Theta_x$ is a pure theta block and the identity in our main theorem remains true after replacing $\vartheta_R$ with $\Theta_x$. This implies the following corollary, which proves the theta block conjecture for all known infinite families of theta blocks of $q$-order 1.

\begin{cor*}[Corollary \ref{cor:MainCorollary}]
    The following infinite series of pure theta blocks of $q$-order 1 satisfy the theta block conjecture.
    \[
    \renewcommand{\arraystretch}{1.2}
    \begin{array}{c|c|c}
    \text{weight} & \text{root system} & \text{theta block} \\ \hline
    & \\[-4mm]
    2&  A_4                                          & \eta^{-6}\vartheta_{a}\vartheta_{b}\vartheta_{c}\vartheta_{d}
    \vartheta_{a+b}\vartheta_{b+c}\vartheta_{c+d}
    \vartheta_{a+b+c}\vartheta_{b+c+d}\vartheta_{a+b+c+d}\\
    & A_1\oplus B_3                          & \eta^{-6}\vartheta_{a}\vartheta_{b}\vartheta_{b+c}\vartheta_{b+2c+2d}
    \vartheta_{b+c+d}\vartheta_{b+c+2d}\vartheta_{c}
    \vartheta_{c+d}\vartheta_{c+2d}\vartheta_{d} \\
    & A_1\oplus C_3                          & \eta^{-6}\vartheta_{a}\vartheta_{b}\vartheta_{2b+2c+d}\vartheta_{b+c}
    \vartheta_{b+2c+d}\vartheta_{b+c+d}\vartheta_{c}
    \vartheta_{2c+d}\vartheta_{c+d}\vartheta_{d}\\
    & B_2\oplus G_2                          & \eta^{-6}\vartheta_{a}
    \vartheta_{a+b}\vartheta_{a+2b}\vartheta_{b}\vartheta_{c}\vartheta_{3c+d}\vartheta_{3c+2d}\vartheta_{2c+d}
    \vartheta_{c+d}\vartheta_{d} \\ \hline
    3& 3A_2                                  & \eta^{-3}\vartheta_{a_1}\vartheta_{a_1+b_1}\vartheta_{b_1}\vartheta_{a_2}\vartheta_{a_2+b_2}              
    \vartheta_{b_2} \vartheta_{a_3}\vartheta_{a_3+b_3}\vartheta_{b_3}   \\      
    &3A_1\oplus A_3                      & \eta^{-3}\vartheta_{a_1}\vartheta_{a_2}\vartheta_{a_3}\vartheta_{a_4}\vartheta_{a_5}              
    \vartheta_{a_6} \vartheta_{a_4+a_5}\vartheta_{a_5+a_6}\vartheta_{a_4+a_5+a_6}   \\    
    &2A_1\oplus A_2\oplus B_2   & \eta^{-3}\vartheta_{a_1}\vartheta_{a_2}\vartheta_{a_3}\vartheta_{a_3+a_4}\vartheta_{a_4}              
    \vartheta_{a_5} \vartheta_{a_5+a_6}\vartheta_{a_5+2a_6}\vartheta_{a_6}  \\ \hline                                                   
    4& 8A_1                                  & \vartheta_{a_1} \vartheta_{a_2}\vartheta_{a_3}\vartheta_{a_4}\vartheta_{a_5}\vartheta_{a_6}     
    \vartheta_{a_7}\vartheta_{a_8}
    \end{array}
    \]
\end{cor*}

The paper is structured as follows. In Sections \ref{section:JacobiForms} and \ref{section:VVMFs} we recall the definitions and some constructions of Jacobi forms of lattice index and modular forms for the Weil representation. In Section \ref{section:OAF} we recall the definition of the Gritsenko lift and of the Borcherds lift. In Section \ref{sec:5-thetablocks} we determine those root systems $R$ for which $\vartheta_R$ has vanishing order $1$ in $q$ and investigate the corresponding lattices $\underline{R}$. In Section \ref{section:Conway} we construct strongly-reflective Borcherds products $\Psi_R$ of singular weight on the maximal even sublattice of $\underline{R}$ and observe that they already appear in the work of Scheithauer. In Section \ref{section:proof} we prove that $G(\vartheta_R)=\Psi_R$ and deduce our main theorem.

\section{Jacobi forms of lattice index}\label{section:JacobiForms}
We denote by $\HH=\{\tau\in \C : \im(\tau)>0 \}$ the complex upper-half plane and for a complex number $z$ we write $e(z)$ for $e^{2\pi i z}$ and we denote by $\sqrt{z}$ the principal branch of the square root. Let $L$ be an integral positive definite lattice with bilinear form $(\cdot,\cdot)$ and $L^\vee$ its dual lattice. The shadow $L^\bullet$ of $L$ is defined by
\[
L^\bullet=\{y\in \Q\otimes L : (x,x)/2=(y,x)\mod \Z \text{ for all } x\in L\}.
\]
Note that $L^\bullet=L^\vee$ if $L$ is even.
\begin{definition}
    For $k\in\frac{1}{2}\Z$ and a character (or multiplier system) $\chi\colon \SL_2(\Z)\to \C^*$ of finite order a holomorphic  function $\varphi : \HH \times (\C \otimes L) \rightarrow \C$ is 
    called a nearly holomorphic Jacobi form of weight $k$, character $\chi$ and index $L$,
    if it satisfies 
    \begin{align*}
    \varphi \left( \frac{a\tau +b}{c\tau + d},\frac{\mathfrak{z}}{c\tau + d} 
    \right)& = \chi\left(\begin{pmatrix}
    a & b \\ c & d
    \end{pmatrix}\right)\sqrt{c\tau + d}^{2k} 
    e{\left(\frac{c(\mathfrak{z},\mathfrak{z})}{2(c 
            \tau + d)}\right)} \varphi ( \tau, \mathfrak{z} ), \quad \forall\begin{pmatrix}
        a & b \\ c & d
    \end{pmatrix} \in \SL_2(\Z), \\
    \varphi (\tau, \mathfrak{z}+ x \tau + y)&=e((x,x)/2+(y,y)/2)e{\bigl( -\tau(x,x)/2 -(x,\mathfrak{z})\bigr)} 
    \varphi (\tau, \mathfrak{z} ), \quad \forall x,y\in L,
    \end{align*}
    and if its Fourier expansion takes the form
    \begin{equation*}
    \varphi ( \tau, \mathfrak{z} )= \sum_{\substack{n\in \Q \\ n\geq n_0 }}\sum_{\ell\in L^\bullet}f(n,
    \ell)q^n\zeta^\ell, \quad q^n=e^{2\pi i n\tau}, \zeta^\ell=e^{2\pi i (\ell,\mathfrak{z})},
    \end{equation*}
    for some constant $n_0$. The coefficients $f(n,\ell)$ with $2n-(\ell,\ell)<0$ are called the singular coefficients. If all singular coefficients vanish,
    then $\varphi$ is called a holomorphic
    Jacobi form. We denote the spaces of nearly holomorphic and holomorphic Jacobi forms of weight $k$, character $\chi$ and index $L$ by $J_{k,L}^!(
   \chi)$ and $J_{k,L}(\chi)$. If the character is trivial, we omit it.
\end{definition}

\begin{remark}
    If $L$ has rank $1$ and determinant $|L^\vee/L|=m$, then the space of Jacobi forms of index $L$ equals the space of  classical Jacobi forms of index $m/2$ introduced in the introduction. 
\end{remark}

In the introduction we have seen that theta blocks are examples of classical Jacobi forms. Similarly, one can try to obtain Jacobi forms of lattice index as products of a power of $\eta$ and of functions of the form $(\tau,\mathfrak{z})\mapsto \vartheta(\tau,(\ell,\mathfrak{z}))$ for $\tau\in\HH$, $\ell\in L^\vee$ and $\mathfrak{z}\in \C\otimes L$. The following theorem gives examples of Jacobi forms of lattice index of this form.

\begin{thm}[{\cite[Theorem 10.1]{GSZ}}]\label{thm:RootThetaFunction}
    Let $R$ be a root system (in the strict sense, see \cite{Humphreys}, \S 9.2) of rank $n$. Let $R^+$ be a system of positive roots of $R$ and let $F$ denote the subset of simple roots in $R^+$.  For $r$ in $R^+$ and $f$ in $F$,  let $\gamma_{r,f}$ be the (non-negative) integers such that $r=\sum_{f\in F}\gamma_{r,f}f$. The function
    \[
    \vartheta_R(\tau,\mathfrak{z}):= \eta(\tau)^{n-N}\prod_{r\in R^+}\vartheta\left(\tau, \sum_{f\in F}\gamma_{r,f} z_f\right)
    \]
    defines an element of $J_{n/2,\underline{R}}(\nu_\eta^{n+2N})$, where $N=\abs{R^+}$, $\mathfrak{z}=(z_f)_{f\in F}\in \C^F$, and the lattice $\underline{R}$ equals $\Z^F$ equipped with the quadratic form $Q(\mathfrak{z})=\frac{1}{2}\sum_{r\in R^+}\left(\sum_{f\in F}\gamma_{r,f} z_f\right)^2$.
\end{thm}

If $\varphi\in J_{k,L}(\chi)$ is a Jacobi form of lattice index, then every non-zero element $x\in L$ can be used to obtain a classical Jacobi form in the following way. Let $K$ be the lattice $\Z$ with bilinear form $(u,v)=muv$, where $m=(x,x)$. We define the embedding $s_x\colon K\to L$ by $s_x(u)=ux$ and  
\[
s_x^*\colon J_{k,L}\to J_{k,K}, \varphi(\tau,\mathfrak{z})\mapsto \varphi(\tau,s_x(w)) \quad (w\in \C\otimes K).
\]
The image is a Jacobi form of index $K$ and we recall that this is the same thing as a classical Jacobi form of index $m/2$. We call the classicial Jacobi form $s_x^*\varphi$ the specialization of $\varphi$ at $x$. By specializing the functions $\vartheta_R$ at an integer vector $x=(x_f)_{f\in F}$ with $x_f\neq 0$ (if one of the $x_f$ equals zero, then $s_x^*\vartheta_R$ vanishes), we obtain a pure theta block
\[
\eta(\tau)^{n-N}\prod_{r\in R^+}\vartheta\left(\tau, z\sum_{f\in F}\gamma_{r,f} x_f\right)\in J_{n/2,Q(x)}(\nu_\eta^{n+2N})
\]
in the variables $(\tau,z)$ in $\HH\times \C$. 

\section{Modular forms for the Weil representation}\label{section:VVMFs}
We recall the definition of a discriminant form. For more details we refer the reader to \cite[Section 2]{ScheithauerWeil}.
A discriminant form is a finite abelian group $D$ with a $\Q/\Z$-valued non-degenerate quadratic form $q\colon D\to \Q/\Z$. We denote by $b\colon D\times D\to \Q/\Z$ the associated bilinear form $b(\gamma_1,\gamma_2)=q(\gamma_1+\gamma_2)-q(\gamma_1)-q(\gamma_2)$. The level of $D$ is the smallest positive integer $N$ such that $Nq(\gamma)=0\mod 1$ for all $\gamma\in D$ and the signature $\sign(D)\in \Z/8\Z$ of $D$ is defined by
\[
\sum_{\gamma\in D}e(q(\gamma))=\sqrt{|D|}e(\sign(D)/8).
\]
 For a positive integer $c$ we define $D_c=\{\gamma\in D : c\gamma=0\}$ and $D^c=\{c\beta : \beta\in D\}$. Then the sequence
 \[
 0\rightarrow D_c\rightarrow D\rightarrow D^c\rightarrow 0
 \]
 is exact. Let $k$ be the largest integer such that $2^k\mid N$. We define the oddity of $D$ to be the signature of $D_{2^k}$. If the signature of $D$ is even, we define a Dirichlet character $\chi_D$ of conductor $N$ by
\[
\chi_D(a)=\legendre{a}{|D|}e((a-1)\odd(D)/8).
\]

If $M$ is an even lattice with dual lattice $M^\vee$, then the reduction $q$ of the quadratic form $x\mapsto (x,x)/2$ on $M^\vee$ modulo $\Z$ turns  $D(M)=M^\vee/M$ into a discriminant form and every discriminant form arises in this way for some even lattice $M$. The level of $D(M)$ coincides with the level of $M$ and the signature of $D$ is equal to the reduction of the signature of $M$ modulo $8$ by Milgram's formula.
\begin{definition}
Let $D$ be a discriminant form of even signature. Let $\C[D]$ be the group ring of $D$ with basis  $\{\textbf{e}_\gamma: \gamma \in D\}$. Then
 \begin{align*}
 \rho_D(T)\textbf{e}_\gamma &= e(-q(\gamma))\textbf{e}_\gamma,\\
 \rho_D(S)\textbf{e}_\gamma &= \frac{e(\sign(D)/8)}{\sqrt{\abs{D}}} \sum_{\beta\in D}e(b(\gamma,\beta))\textbf{e}_\beta
 \end{align*}
 defines a representation of $\SL_2(\Z)$ on $\C[D]$. This representation is called the Weil representation associated to $D$.
\end{definition}

\begin{definition}
Let $F(\tau)=\sum_{\gamma\in D}F_\gamma(\tau)\textbf{e}_\gamma$ be a holomorphic function on $\HH$ with values in $\C[D]$ and $k\in\Z$. The function $F$ is called a nearly holomorphic modular form of weight $k$ for $\rho_D$ if 
\[
F\Big(\frac{a\tau +b}{c\tau +d}\Big)=(c\tau+d)^k\rho_D(A)F(\tau), \quad \forall A=\begin{pmatrix}
a & b \\ c & d
\end{pmatrix} \in \SL_2(\Z)
\]
and if $F$ has a Fourier expansion of the form
\[
F(\tau)=\sum_{\gamma\in D}\sum_{\substack{n\in \Z-q(\gamma)\\ n\geq n_0}}c_\gamma(n)q^n\textbf{e}_\gamma.
\]
The sum 
$\sum_{\gamma\in D}\sum_{n<0}c_\gamma(n)q^n\textbf{e}_\gamma$
is called the principal part of $F$. If the principal part vanishes, then $F$ is called holomorphic.

\end{definition}

\begin{remark}
The orthogonal group $\Orth(D)$ acts on $\C [D]$ via 
$
\sigma\left(\sum_{\gamma\in D}a_\gamma \textbf{e}_\gamma  \right)=\sum_{\gamma\in D}a_\gamma \textbf{e}_{\sigma(\gamma)}
$
and this action commutes with that of $\rho_D$ on $\C [D]$. Thus $\Orth(D)$ acts on modular forms for the Weil representation. 

\end{remark}

One way to obtain vector-valued modular forms is given by the following proposition.
\begin{prp}[{\cite[Theorem 3.1]{Scheithauer-ModForms}}]
    Let $f$ be a scalar-valued modular form of weight $k$ and character $\chi_D$ for $\Gamma_0(N)$. Let $S$ be an isotropic subset of $D$ which is invariant under $(\Z/N\Z)^*$. Then
    \[
    F_{\Gamma_0(N),f,S}(\tau)=\sum_{M\in \Gamma_0(N)\backslash\SL_2(\Z)}\sum_{\gamma\in S}f| M(\tau)\rho_D(M^{-1})\textbf{e}_\gamma
    \]
    is a vector-valued modular form of weight $k$ for $\rho_D$. The function $F_{\Gamma_0(N),f,S}$ is invariant under the automorphisms of $D$ which stabilize $S$.
\end{prp}

Suppose $L$ is an even positive definite lattice with discriminant form $D(L)$. The theta series $\Theta_{\gamma}^L\colon \HH\times (\C\otimes L)\to \C$ associated to $L$ and $\gamma\in D=D(L)$ is defined by
\[
\Theta_{\gamma}^{L}(\tau,\mathfrak{z})=\sum_{\ell \in \gamma +L}q^{(\ell,\ell)/2}\zeta^{\ell}, \quad \gamma\in D.
\]
The map 
\begin{equation}\label{eq:Correspondence}
F(\tau)=\sum_{\gamma\in D}F_\gamma(\tau)\textbf{e}_\gamma \longmapsto \sum_{\gamma\in D}F_\gamma(\tau)\Theta_\gamma^L(\tau,\mathfrak{z})
\end{equation}
defines an isomorphism between the spaces of nearly holomorphic modular forms of weight $k$ for $\rho_{D}$ and of nearly holomorphic Jacobi forms of weight $k+\rk(L)/2$ and index $L$. The principal part of $F$ corresponds to the singular Fourier coefficients of the Jacobi form. Hence the map also induces an isomorphism between the subspaces of holomorphic modular forms for $\rho_{D}$ and holomorphic Jacobi forms of index $L$.

\section{Automorphic forms on orthogonal groups}\label{section:OAF}

Let $M$ be an even lattice of signature $(2,n)$ with $n\geq 3$. The Hermitian symmetric domain of type IV attached to $M$ is defined as (we choose one of the two connected components)
\begin{equation*}
\cD(M)=\{[\mathcal{Z}] \in  \PP(\C\otimes M):  (\mathcal{Z}, \mathcal{Z})=0, (\mathcal{Z},\bar{\mathcal{Z}}) > 0\}^{+}.
\end{equation*}
Let $\Orth^+(M) \subset \Orth(M)$ be the index $2$ subgroup preserving the component $\cD(M)$. The discriminant kernel $\widetilde{\Orth}^+(M)$ is the kernel of the natural homomorphism $\Orth^+ (M) \to \Orth(D(M))$. Let $\Gamma$ be a finite index subgroup of $\Orth^+ (M)$ and $k\in \Z$. A modular form of weight $k$ and character $\chi\colon \Gamma\to \C^*$ for $\Gamma$ is a meromorphic function $F\colon \cD(M)^{\bullet}\to \C$ on the affine cone $\cD(M)^{\bullet}$ over $\cD(M)$ satisfying
\begin{align*}
F(t\mathcal{Z})&=t^{-k}F(\mathcal{Z}), \quad \forall t \in \C^*,\\
F(g\mathcal{Z})&=\chi(g)F(\mathcal{Z}), \quad \forall g\in \Gamma.
\end{align*}

If $F$ is holomorphic, then it either has weight 0 in which case it is constant, or has weight at least $n/2-1$ (see \cite[Corollary 3.3]{Bor95}). The minimal possible positive weight $n/2-1$ is called the singular weight.

For any negative norm vector $v\in M^\vee$, we define the rational quadratic divisor associated to $v$ as
\[
\cD_v(M)=v^\perp\cap \cD(M)=\{ [\mathcal{Z}]\in \cD(M) : (\mathcal{Z},v)=0\}. 
\]
We say that a holomorphic orthogonal modular form $F$ for $\widetilde{\Orth}^+(M)$ is reflective if its divisor is a union of divisors of the form $\cD_v(M)$ for roots $v\in M^\vee$ (a root is a primitive vector $v\in M^\vee$ such that the reflection $x\mapsto x-2(x,v)v/(v,v)$ at $v^\perp$ maps $M$ to $M$) and we say that $F$ is strongly-reflective, if in addition the multiplicities of all zeros are $1$.

In his famous paper \cite{Bor98}, Borcherds described the following way to construct orthogonal modular forms with zeros and poles on rational quadratic divisors from vector-valued modular forms. Since they have an infinite product expansion at every $0$-dimensional cusp, they are called Borcherds products. 
\begin{thm}[{\cite[Theorem 13.3]{Bor98}}]\label{Borcherds13.3}
    Let $M$ be an even lattice of signature $(2,n)$, $n\geq 3$. Let $D$ be the discriminant form of $M(-1)$. Let
    \[
    F=\sum_{\gamma\in D}\sum_{m\in \Z-q(\gamma)}c_\gamma(m)q^m\textbf{e}_\gamma
    \] 
    be a nearly holomorphic modular form of weight $1-n/2$ for $\rho_D$ with integral Fourier coefficients $c_\gamma(m)$ for all $m\leq 0$. Then there is a meromorphic function $\Psi\colon\cD(M)^\bullet\to \C$ with the following properties.
    \begin{enumerate}
        \item $\Psi$ is a modular form of weight $c_0(0)/2$ for the group $\Orth(M,F)^+=\{\sigma\in \Orth^+(M) : \sigma(F)=F\}$ and some multiplier system $\chi$ of finite order. If $c_0(0)$ is even, then $\chi$ is a character.
        \item  The only zeros or poles of $\Psi$ lie on rational quadratic divisors $\cD_v(M)$, where $v$ is a primitive vector of negative norm in $M^\vee$. The divisor $\cD_v(M)$ has order
        \begin{equation}\label{DivisorAutProd}
        \sum_{m\in \Z_{>0}} c_{mv}(m^2(v,v)/2).
        \end{equation}
        \item For each primitive norm $0$ vector $z\in M$, an associated vector $z'\in M^\vee$ with $(z,z')=1$ and for each Weyl chamber $W$ of $K=L/\Z z \cong M\cap z^\perp\cap {z'}^\perp$ with $L=M\cap z^\perp$, the restriction $\Psi_z$ has an infinite product expansion converging when $Z$ is in a neighbourhood of the cusp $z$ and $\im (Z)\in W$ which is some constant times
        \[
        e((Z,\rho))\prod_{\substack{\lambda\in K^\vee\\(\lambda,W)>0}}\prod_{\substack{\delta\in M^\vee/M \\ \delta|L=\lambda}}(1-e((\lambda,Z)+(\delta,z')))^{c_\delta((\lambda,\lambda)/2)}.
        \] 
    \end{enumerate}
\end{thm}

For the rest of this section we assume that $M$ splits two hyperbolic planes, i.e. $M=U\oplus U_1\oplus L(-1)$, where $U=\Z e\oplus\Z f$ ($(e,e)=(f,f)=0$, $(e,f)=1$), $U_1=\Z e_1\oplus\Z f_1$ and $L$ is an even positive definite lattice. We choose $(e,e_1,\ldots,f_1,f)$ as a basis of $M$. 
Here $\ldots$ denotes a basis of $L(-1)$.  

Every $[\mathcal{Z}]\in \cD(M)$ has a unique representative of the form $(\ast,\tau,\mathfrak{z},\omega,1)\in \cD(M)^\bullet$ with $\tau,\omega\in \HH$ and $\mathfrak{z}\in \C\otimes L$. Therefore, at the one-dimensional cusp determined by the isotropic plane $\latt{e,e_1}$, the symmetric space $\cD(M)$ can be realized as the tube domain 
\[\cH(L)=\{Z=(\tau,\mathfrak{z},\omega)\in \HH\times (\C\otimes L)\times \HH: 
(\im Z,\im Z)>0\}, 
\]
where $(\im Z,\im Z)=2\im \tau \im \omega - 
(\im \mathfrak{z},\im \mathfrak{z})_L$. In this realization an orthogonal modular form $F$ of weight $k$ and trivial character for $\widetilde{\Orth}^+(M)$ has a Fourier-Jacobi expansion
\[
F(\tau,\mathfrak{z},\omega)=\sum_{m\in \Z_{\geq 0}}\varphi_m(\tau,\mathfrak{z})e(m\omega)
\] 
where $\varphi_m$ is a Jacobi form of weight $k$ and index $L(m)$.

The Gritsenko lift associates an orthogonal modular form to a Jacobi form of lattice index.
\begin{thm}[{\cite[Theorem 3.1]{Gri94}}]
    Let $k$ be integral and $\varphi \in J_{k,L}$. For a positive integer $m$, we let
    \[
    \varphi | T_{-}(m)(\tau,
    \mathfrak{z})=m^{-1}\sum_{\substack{ad=m,a>0\\ 0\leq b <d}}a^k \varphi
    \left(\frac{a\tau+b}{d},a\mathfrak{z}\right).
    \]
    Then the function
    \[ G(\varphi)(Z)=f(0,0)G_k(\tau)+\sum_{m\geq 1}\varphi |
    T_{-}(m)(\tau,\mathfrak{z}) e(m\omega)
    \]
    is a modular form of weight $k$ and trivial character for $\widetilde{\Orth}^+(2U\oplus L(-1))$. Moreover, this modular form is symmetric, i.e.
    $\Grit(\varphi)(\tau,\mathfrak{z}, \omega)=
    \Grit(\varphi)(\omega,\mathfrak{z}, \tau)$. Here $f(0,0)$ is the zeroth Fourier coefficient of $\varphi$ and $G_k$ is the Eisenstein series of weight $k$, normalized such that the Fourier coefficient at $q$ is $1$.
\end{thm}

\begin{remark}\label{rmk:ZerosOfGritsenko}
    Let $\ell$ be a non-zero vector in $L^\vee$ such that $\varphi$ vanishes on
    \[
    \{(\tau,\mathfrak{z})\in \HH\times (\C\otimes L) : (\ell,\mathfrak{z})\in \Z\tau+\Z\}.
    \]
    Then the same is true for $\varphi\lvert_{k}T_{-}(m)$ for every $m\geq 1$. Therefore, if $f(0,0)=0$, then $G(\varphi)$ vanishes on $\cD_v(M)$ for every $v\in M^\vee$ of the form $v=(0,0,\ell,n,0)$ with $n\in \Z$. 
\end{remark}
\begin{remark}
    Using \eqref{eq:Correspondence}, the Gritsenko lift can also be described in terms of vector-valued modular forms instead of Jacobi forms. In this setting the Gritsenko lift is known as the additive Borcherds lift, which also exists if $M$ does not split two hyperbolic planes (see \cite[Theorem 14.3]{Bor98}).
\end{remark}

Using the correspondence between Jacobi forms and vector-valued modular forms, we can also describe Theorem \ref{Borcherds13.3} in terms of Jacobi forms.
\begin{thm}[{\cite[Theorem 4.2]{Gritsenko-Reflective}}]
    Let $L$ be an even positive definite lattice. Let
    \[
    \varphi(\tau,\mathfrak{z})=\sum_{n\in \Z,\ell\in L^\vee}f(n,\ell)q^n\zeta^\ell\in J_{0,L}^!
    \]
    with $f(n,\ell)\in \Z$ for all $2n-(\ell,\ell)\leq 0$. There is a meromorphic modular form of weight 
    $f(0,0)/2$ and character $\chi$ with respect to  
    $\widetilde{\Orth}^+(2U\oplus L(-1))$ defined as
    \begin{equation}\label{eq:JacobiLift}
    B(\varphi)=
    \biggl(\Theta_{f(0,\ast)}
    (\tau,\mathfrak{z})e^{2\pi i\, C\omega}\biggr)
    \exp \left(-G(\varphi)\right),
    \end{equation}
    where
    $C=\frac{1}{2\rk(L)}\sum_{\ell\in L^\vee}f(0,\ell)(\ell,\ell)$ and
    \begin{equation*}\label{FJtheta}
    \Theta_{f(0,\ast)}(\tau,\mathfrak{z})
    =\eta(\tau)^{f(0,0)}\prod_{\ell >0}
    \biggl(\frac{\vartheta(\tau,(\ell,\mathfrak{z}))}{\eta(\tau)} 
    \biggr)^{f(0,\ell)}
    \end{equation*}
    is a theta block.
    The character $\chi$ is induced by the character of the theta block
    and by the relation $\chi(V)=(-1)^D$, where
    $V\colon (\tau,\mathfrak{z}, \omega) \mapsto (\omega,\mathfrak{z},\tau)$, 
    and $D=\sum_{n<0}\sigma_0(-n) f(n,0)$.
    
    The poles and zeros of $B(\varphi)$ lie on the rational quadratic 
    divisors $\cD_v$, where $v\in 2U\oplus L^\vee(-1)$ is a primitive vector 
    with $(v,v)<0$. The multiplicity of this divisor is given by 
    \[ \mult \cD_v = \sum_{d\in \Z_{>0} } f(d^2n,d\ell),\]
    where $n\in\Z$, $\ell\in L^\vee$ such that 
    $(v,v)=2n-(\ell,\ell)$ and $v-(0,0,\ell,0,0)\in 2U\oplus L(-1)$.
    
\end{thm} 

\begin{remark}[{\cite[Corollary 4.3]{Gritsenko-Reflective}}]\label{rmk:ShapeOfBorcherds}
From \eqref{eq:JacobiLift}, we see that the Fourier-Jacobi expansion of $B(\varphi)$ at the one-dimensional cusp determined by the decomposition $M=2U\oplus L(-1)$ is given by
\[
B(\varphi)(\tau,\mathfrak{z},\omega)=\Theta_{f(0,\ast)}(\tau,\mathfrak{z})e^{2\pi i C\omega}\left(1-\varphi(\tau,\mathfrak{z})e^{2\pi i \omega}+\frac{1}{2}\big(\varphi^2(\tau,\mathfrak{z})-\varphi|_0T_-(2)(\tau,\mathfrak{z})\big)e^{4\pi i \omega}+\ldots\right).
\]
In particular, we see that if the Gritsenko lift $G(\vartheta_R)=\vartheta_R e^{2\pi i \omega}+\vartheta_R|T_-(2)e^{4\pi i \omega}+\ldots$ is a Borcherds product $B(\varphi)$, then $C=1$, $\Theta_{f(0,\ast)}=\vartheta_R$ and
\[
\varphi=-\frac{\vartheta_R|T_-(2)}{\vartheta_R}.
\]
\end{remark}

\section{Theta blocks related to root systems}\label{sec:5-thetablocks}
The theta block conjecture mentioned in the introduction states that every pure theta block $\Theta$ with order of vanishing $1$ in $q$ satisfies $G(\Theta)=B(-\frac{\Theta | T_{-}(2)}{\Theta})$. Recall from Section \ref{section:JacobiForms} that one way to obtain theta blocks is by specializing the Jacobi forms $\vartheta_R$ from Theorem \ref{thm:RootThetaFunction}. The following theorem, which we prove in Section \ref{section:proof}, implies that the theta block conjecture is true for theta blocks obtained in this way.
\begin{thm}\label{thm:mainThm}
 Let $R$ be a root system and let $\vartheta_R$ be as in Theorem \ref{thm:RootThetaFunction}. Suppose that $\vartheta_R$ has vanishing order $1$ in $q$. Then 
 \[
 G(\vartheta_R)=B\Big(-\frac{\vartheta_R | T_{-}(2)}{\vartheta_R}\Big).
 \]
\end{thm}

We first determine those root systems for which $\vartheta_R$ has vanishing order $1$ in $q$.
\begin{prp}\label{prop:rootSystems}
Let $R$ be a root system such that $\vartheta_R$ has $q$-order $1$. Then $R$ is one of the following root systems.\[
\renewcommand{\arraystretch}{1.2}
\begin{array}{c|c}
\text{weight} & \text{root systems} \\ \hline
& \\[-4mm]
2            & A_4,  A_1\oplus B_3, A_1\oplus C_3, B_2 \oplus G_2 \\
3            &  3 A_2, 3A_1\oplus A_3, 2A_1\oplus A_2\oplus B_2\\
4            & 8A_1
\end{array}
\]
\end{prp}
\begin{proof}
Since $\eta$ has $q$-order $1/24$ and $\vartheta$ has $q$-order $1/8$, the function $\vartheta_R$ has $q$-order $n/24+N/12$. Therefore, we obtain the condition $n+2N=24$. We see that $n$ and $N$ are bounded. The root system $R$ can be decomposed into irreducible root systems of type $A_n$, $B_n$, $C_n$, $D_n$, $E_6$, $E_7$, $E_8$, $F_4$ and $G_2$, so there are only finitely many possibilities for $R$ which can be checked by hand.
\end{proof}
Specializing these Jacobi forms of lattice index yields the infinite series of theta blocks with $q$-order $1$ given in Table \ref{table}.
\begin{table}[h]\caption{Theta blocks of $q$-order 1}\label{table}
 \[
\renewcommand{\arraystretch}{1.2}
\begin{array}{c|c|c}
\text{weight} & \text{root system} & \text{theta block} \\ \hline
& \\[-4mm]
2&  A_4                                          & \eta^{-6}\vartheta_{a}\vartheta_{b}\vartheta_{c}\vartheta_{d}
\vartheta_{a+b}\vartheta_{b+c}\vartheta_{c+d}
\vartheta_{a+b+c}\vartheta_{b+c+d}\vartheta_{a+b+c+d}\\
& A_1\oplus B_3                          & \eta^{-6}\vartheta_{a}\vartheta_{b}\vartheta_{b+c}\vartheta_{b+2c+2d}
\vartheta_{b+c+d}\vartheta_{b+c+2d}\vartheta_{c}
\vartheta_{c+d}\vartheta_{c+2d}\vartheta_{d} \\
& A_1\oplus C_3                          & \eta^{-6}\vartheta_{a}\vartheta_{b}\vartheta_{2b+2c+d}\vartheta_{b+c}
\vartheta_{b+2c+d}\vartheta_{b+c+d}\vartheta_{c}
\vartheta_{2c+d}\vartheta_{c+d}\vartheta_{d}\\
& B_2\oplus G_2                          & \eta^{-6}\vartheta_{a}
\vartheta_{a+b}\vartheta_{a+2b}\vartheta_{b}\vartheta_{c}\vartheta_{3c+d}\vartheta_{3c+2d}\vartheta_{2c+d}
\vartheta_{c+d}\vartheta_{d} \\ \hline
3& 3A_2                                  & \eta^{-3}\vartheta_{a_1}\vartheta_{a_1+b_1}\vartheta_{b_1}\vartheta_{a_2}\vartheta_{a_2+b_2}              
\vartheta_{b_2} \vartheta_{a_3}\vartheta_{a_3+b_3}\vartheta_{b_3}   \\      
&3A_1\oplus A_3                      & \eta^{-3}\vartheta_{a_1}\vartheta_{a_2}\vartheta_{a_3}\vartheta_{a_4}\vartheta_{a_5}              
\vartheta_{a_6} \vartheta_{a_4+a_5}\vartheta_{a_5+a_6}\vartheta_{a_4+a_5+a_6}   \\    
&2A_1\oplus A_2\oplus B_2   & \eta^{-3}\vartheta_{a_1}\vartheta_{a_2}\vartheta_{a_3}\vartheta_{a_3+a_4}\vartheta_{a_4}              
\vartheta_{a_5} \vartheta_{a_5+a_6}\vartheta_{a_5+2a_6}\vartheta_{a_6}  \\ \hline                                                   
4& 8A_1                                  & \vartheta_{a_1} \vartheta_{a_2}\vartheta_{a_3}\vartheta_{a_4}\vartheta_{a_5}\vartheta_{a_6}     
\vartheta_{a_7}\vartheta_{a_8}
\end{array}
\]
\end{table}
\begin{remark}
As explained above, we prove the theta block conjecture in the case of a pure theta block $\Theta$ obtained by specializing one of the functions $\vartheta_R$. In fact, every pure theta block of $q$-order $1$ has weight less than $12$ and every pure theta block of weight $4\leq k\leq 11$ is of the form 
\[
\eta^{3t}\prod_{j=1}^{8-t}\vartheta_{a_j}, \quad 0\leq t \leq 7,
\]
and is therefore related to the infinite family of type $8A_1$, so the theta block conjecture is true for weights $k\geq 4$ without the condition that $\Theta$ is a specialization of some $\vartheta_R$ (see \cite[Theorem 8.2]{GPY}).  However,  for weights $2$ and $3$ there do exist theta blocks of $q$-order $1$ not in any of the families given in Table \ref{table},  such as the weight $2$ theta blocks of index $49$ in \cite[Table 7]{GSZ} and the weight $3$ theta block $\vartheta_2^2\vartheta_3^5\vartheta_5\vartheta_6/\eta^{3}$ .
\end{remark}
If the root system $R$ is irreducible, there is also the following description of the lattice $\underline{R}$, which will be more useful for our purposes  (cf. \cite[Section 10]{GSZ}).

Let $R^\vee$ be the dual root system of $R$, i.e.
\[
R^\vee=\left\{ \frac{2}{(r,r)}r: r\in R \right\}.
\] 
The weight lattice of $R^\vee$ is
\[
\Lambda(R^\vee)=\{ v\in \Q\otimes R: (v,r)\in \Z, \; \forall \; r\in R \}. 
\]
With the definition $h=\frac{1}{n}\sum_{r\in R^+}(r,r)$, the identity 
\[
\sum_{r\in R^+}(r,\mathfrak{z})^2=h(\mathfrak{z},\mathfrak{z})
\]
holds. Let $\{w_f\}_{f\in F}$ denote the fundamental weights of $R^\vee$, i.e. the dual basis of $F$. We let $L$ be the integral lattice $\Lambda(R^\vee)(h)$, i.e. $L$ is the $\Z$-module $\Lambda(R^\vee)$ with bilinear form $\langle v,w\rangle=h(v,w)$. Then $v\mapsto \sum v_fw_f$, $v=(v_f)_{f\in F}\in \Z^F$ defines an isomorphism between $\underline{R}$ and $L$. The function $\vartheta_R$ then takes the form
\[
\vartheta_R(\tau,\mathfrak{z})=\eta(\tau)^{n-N}\prod_{r\in R^+}\vartheta(\tau,\langle r/h,\mathfrak{z}\rangle)
\]
for all $\mathfrak{z}\in \C\otimes L$.

If $R$ is reducible, we can decompose $R$ into a direct sum of irreducible root systems and the lattice $\underline{R}$ is then isomorphic to the direct sum of the corresponding lattices $L$.

The following table gives the lattice $L$ and its maximal even sublattice $L_{\text{ev}}$ for all root systems $R$ from Proposition \ref{prop:rootSystems}. We also list the genus of $L_{\text{ev}}$.  We refer to \cite{CS99} or \cite[\S 3]{Scheithauer-Class} for the description of genera of lattices.

\[
\renewcommand{\arraystretch}{1.2}
\begin{array}{c|c|c|c|c}
\text{weight}  & R & L & L_{\text{ev}} & \text{ genus of } L_{\text{ev}}\\ \hline
& \\[-4mm]
2&  A_4                              & A_4^\vee(5)           & A_4^\vee(5) & \II_{4,0}(5^{+3})\\
&  A_1\oplus B_3                & \Z\oplus \Z^3(5)         &  L_4 & \II_{4,0}(2_{\II}^{+2}5^{+3})\\
&   A_1\oplus C_3               &\Z\oplus A_3^\vee(8)         & A_1(2)\oplus A_3^\vee(8) & \II_{4,0}(2_3^{-1}4_1^{+1}8_{\II}^{-2})\\
&   B_2\oplus G_2                & \Z^2(3)\oplus A_2(4)       & 2A_1(3) \oplus A_2(4 ) & \II_{4,0}(2_6^{+2}4_{\II}^{-2}3^{-3})\\  \hline
3&  3A_2                              & 3A_2          &  3A_2 & \II_{6,0}(3^{-3}) \\                                               
&  3A_1\oplus A_3                & \Z^3\oplus A_3^\vee(4)       & S_6 & \II_{6,0}(2_6^{+2}4_{\II}^{-2})\\    
&  2A_1\oplus A_2\oplus B_2  &\Z^2\oplus A_2\oplus \Z^2(3)    & L_6 & \II_{6,0}(2_{\II}^{+2}3^{-3})\\   \hline                                               
4&  8A_1                                  & \Z^8      & D_8 & \II_{8,0}(2_{\II}^{+2})
\end{array}
\]
where $L_4$, $S_6$ and $L_6$ have the following Gram matrices:
\[
L_4=\left(\begin{array}{cccc}
4 & 2 & 2 & 2 \\ 
2 & 6 & 1 & 1 \\ 
2 & 1 & 6 & 1 \\ 
2 & 1 & 1 & 6
\end{array}  \right),
\]
\begin{align*}
S_6=\left( \begin{array}{cccccc}
2 & 0 & 1 & 1 & 1 & 0 \\ 
0 & 2 & 1 & 1 & 1 & 0 \\ 
1 & 1 & 4 & 2 & 2 & 3 \\ 
1 & 1 & 2 & 4 & 0 & 1 \\ 
1 & 1 & 2 & 0 & 4 & 1 \\ 
0 & 0 & 3 & 1 & 1 & 4
\end{array} \right), \qquad
L_6=\left( \begin{array}{cccccc}
4 & 2 & 0 & 0 & -2 & 0 \\ 
2 & 4 & 0 & 0 & -1 & 0 \\ 
0 & 0 & 2 & -1 & 0 & 0 \\ 
0 & 0 & -1 & 2 & 0 & 0 \\ 
-2 & -1 & 0 & 0 & 2 & 1 \\ 
0 & 0 & 0 & 0 & 1 & 4
\end{array} \right).
\end{align*}

For every $L_\text{ev}$ in the above table, its genus contains only one class. Thus the lattice $M=2U\oplus L_{\text{ev}}(-1)$ has only one model splitting two hyperbolic planes.  We view the Gritsenko lifts and Borcherds products in Theorem  \ref{thm:mainThm} as  orthogonal modular forms on $2U\oplus L_{\text{ev}}(-1)$.

\begin{remark}\label{rmk:Roots}
    Suppose once again that $R$ is an irreducible root system. The function $\vartheta_R\colon \HH\times (\C\otimes L)\to \C$ is a Jacobi form of index $L$ and hence also a Jacobi form of index $L_{\text{ev}}$. Since the Jacobi theta function $\vartheta(\tau,z)$ vanishes for $z\in \Z\tau+\Z$, we see that $\vartheta_R$ vanishes along the divisor
    \[
   \{(\tau,\mathfrak{z})\in \HH\times (\C\otimes L_{\text{ev}}) : \langle r/h,\mathfrak{z}\rangle \in \Z\tau+\Z \},
   \]
   where $r$ is a positive root in $R$. The element $\lambda=r/h$ is an element of $L^\vee\subset L_{\text{ev}}^\vee$ and $\langle \lambda,\lambda\rangle=(r,r)/h$. In the following table we list the order of $\lambda$ in $L^\vee/L$ and in $L_{\text{ev}}^\vee/L_{\text{ev}}$ and the norm $\langle \lambda,\lambda\rangle$. This depends on whether $r$ is a long or a short root of $R$.

 \[
\renewcommand{\arraystretch}{1.2}
\begin{array}{c|c|c|c|c}
R  & \text{long or short root} & \ord(\lambda)\text{ in } L^\vee/L & \ord(\lambda) \text{ in }  L_{\text{ev}}^\vee/L_{\text{ev}} & \langle \lambda,\lambda\rangle\\ \hline
& \\[-4mm]
 A_1 &                             & 1           & 2 & 1\\
 A_n (n>1)&              & n+1         &  n+1 & 2/(n+1) \\
 B_n (n>1)&   \text{ short root }               & 1+2(n-1)         & 2+4(n-1) & 1/(1+2(n-1))\\
          & \text{ long root } & 1+2(n-1) & 1+2(n-1) & 2/(1+2(n-1)) \\
C_n (n>2)&   \text{ short root }     & 4+2(n-1)        & 4+2(n-1) & 2/(4+2(n-1))\\
&  \text{ long root }                               & 2+(n-1)        &  2+(n-1) & 2/(2+(n-1)) \\                                               
G_2&  \text{ short root }                & 12       & 12 & 1/6\\    
&  \text{ long root }  & 4    & 4 & 1/2
\end{array}
\]
\end{remark}

\section{Borcherds products related to Conway's group}\label{section:Conway}

Let $R$ be one of the root systems from the previous section and let $L$ and $L_{\text{ev}}$ as before. Let $M=2U\oplus L_\text{ev}(-1)$. We show that in all of these cases, there exists a strongly-reflective Borcherds product $\Psi_R$ of singular weight on $M$, which can be constructed explicitly as described in the following theorem.
\begin{thm}\label{thm:ReflectiveProducts}
    Let $R$ be one of the root systems from Proposition \ref{prop:rootSystems} and $M=2U\oplus L_{\text{ev}}(-1)$. There exists a strongly-reflective modular form $\Psi_R$ of singular weight for the full modular group $\Orth^+(M)$. This function is the Borcherds product corresponding to the following vector-valued modular form $F$ for the Weil representation $\rho_D$ associated to $D=D(L_\text{ev})$.
    \[
    \renewcommand{\arraystretch}{1.2}
    \begin{array}{c|c|c}
    \text{weight}  & R & F \\ \hline
    & \\[-4mm]
    2&  A_4                              & F_{\Gamma_0(5),5\eta_{1^{-5}5^1},0}\\
    &  A_1\oplus B_3                & F_{\Gamma_0(10),\eta_{1^{-1}2^{-2}5^{-3}10^2},0} \\
    &   A_1\oplus C_3               & F_{\Gamma_0(8),2\eta_{1^{-2}2^{-1}4^{-3}8^2},0}+F_{\Gamma_0(8),8\eta_{1^{-6}2^14^{-5}8^6}+{\eta_{1^{-2}2^{-1}4^{-3}8^2}},D^4}\\
    &   B_2\oplus G_2                & F_{\Gamma_0(12),\eta_{1^{-1}3^{-1}4^{-2}6^{-2}12^2},0}+F_{\Gamma_0(12),\eta_{1^{-4}4^16^{-2}12^1},D^6}\\ [1mm]  \hline
    3&  3A_2                              & F_{\Gamma_0(3),9\eta_{1^{-9}3^3},0}\\                                               
    &  3A_1\oplus A_3                &  F_{\Gamma_0(4),4\eta_{1^{-4}2^{-6}4^{4}},0}+F_{\Gamma_0(4),-2\eta_{1^{-4}2^{-6}4^4},D^2}\\ 
    &  2A_1\oplus A_2\oplus B_2  &   F_{\Gamma_0(6),\eta_{1^{-1}2^{-4}3^{-5}6^4},0} \\  [1mm]  \hline                                               
    4&  8A_1                                  &  F_{\Gamma_0(2),16\eta_{1^{-16}2^8},0}
    \end{array}
    \]   
     
    We write $F=\sum_{\gamma\in D}F_\gamma\textbf{e}_\gamma$. If the level $N$ of $M$ is square-free, then the Fourier expansion of $F_\gamma$ is given by
    \[
    F_\gamma=
    \begin{cases}
    \rk(R) +O(q) &\text{ if } \gamma=0, \\
    q^{-1/d}+O(q^{1-1/d})  & \text{ if } \ord(\gamma)=d \text{ and } q(\gamma)=1/d\mod 1 \text{ for a divisor } d>1 \text{ of } N, \\ 
    O(1) & \text{ in all other cases.}
    \end{cases}    
    \]
    
    The cases where $N$ is not square-free are $R=A_1\oplus C_3$, $B_2\oplus G_2$ and $3A_1\oplus A_3$. 
    
    In the case $R=A_1\oplus C_3$, the discriminant form $D$ is given by $D=2_3^{-1}4_1^{+1}8_{\II}^{-2}$. We let $x_2$ and $x_4$ be the unique elements of order $2$ in $2_3^{-1}$ and $4_1^{+1}$. Then the Fourier expansion of $F_\gamma$ is given by    
    \[
    F_\gamma=
    \begin{cases}
    4+O(q) &\text{ if } \gamma=0, \\
    q^{-1/2}+O(q^{1/2})  & \text{ if } \gamma=x_4,\\ 
    q^{-1/4}+O(q^{3/4}) & \text{ if } \gamma\in \{x_2+2\delta : \delta\in D, 2q(\delta)+b(x_2,\delta)=3/4\mod 1\}, \\
    q^{-1/8}+O(q^{7/8}) & \text{ if } \ord(\gamma)=8 \text{ and } q(\gamma)=1/8\mod 1, \\
    O(1) & \text{ in all other cases.}
    \end{cases}    
    \]
    
    In the case $R=B_2\oplus G_2$, the discriminant form $D$ is given by $D=2_6^{+2}4_{\II}^{-2}3^{-3}$. We let $x_2$ be the unique element with $q(x_2)=1/2\mod 1$ in $2_{6}^{+2}$. Then the Fourier expansion of $F_\gamma$ is given by    
    \[
    F_\gamma=
    \begin{cases}
    4+O(q) &\text{ if } \gamma=0, \\
    q^{-1/2}+O(q^{1/2})  & \text{ if } \gamma=x_2,\\ 
    q^{-1/d}+O(q^{1-1/d}) & \text{ if } \ord(\gamma)=d \text{ and } q(\gamma)=1/d\mod 1 \text{ for some } d\in\{3,4,12\}, \\ 
    q^{-1/6}+O(q^{5/6}) & \text{ if } \gamma\in\{x_2+\delta : \delta\in 3^{-3}, q(\delta)=2/3\mod 1\}, \\
    O(1) & \text{ in all other cases.}
    \end{cases} 
    \]   
    
    In the case $R=3A_1\oplus A_3$, the discriminant form $D$ is given by $D=2_6^{+2}4_{\II}^{-2}$. We let $x_2$ be the unique element with $q(x_2)=1/2\mod 1$ in $2_{6}^{+2}$. Then the Fourier expansion of $F_\gamma$ is given by    
    \[
    F_\gamma=
    \begin{cases}
    6+O(q) &\text{ if } \gamma=0, \\
    q^{-1/2}+O(q^{1/2})  & \text{ if } \gamma=x_2,\\ 
    q^{-1/4}+O(q^{3/4}) & \text{ if } \ord(\gamma)=4 \text{ and } q(\gamma)=1/4\mod 1, \\ 
    O(1) & \text{ in all other cases.}
    \end{cases}       
    \]
\end{thm}
\begin{proof}
    The Fourier expansion of $F$ can be calculated using the formula in \cite[Theorem 3.2]{Scheithauer-ModForms}. The Fourier coefficients of the principal part of $F$ are non-negative integers, so Theorem \ref{Borcherds13.3} yields a holomorphic Borcherds product $\Psi_R$. We note that $F$ is invariant under $\Orth(D)$ by construction, hence the modular form $\Psi_R$ is modular for the full group $\Orth^+(M)$. In all cases, the constant coefficient of $F_0$ is given by $\rk(R)$, so $\Psi_R$ has weight $\rk(R)/2$, which is the singular weight. The divisor of $\Psi_R$ is determined by the principal part of $F$. In all cases, the only contributions to the principal part are terms of the form $q^{-1/d}$ in components $F_\gamma$ with $\gamma\in D$ of order $d$ and $q(\gamma)=1/d\mod 1$ for some divisor $d$ of $N$. This implies that $\Psi_R$ is strongly-reflective (cf. \cite[Section 9]{Scheithauer-Class}).
\end{proof}

\begin{remark}
    The cases where the level of $M$ is square-free can be found in the table at the end of Section 10 in \cite{Scheithauer-Class}. We reconstruct them at the standard 1-dimensional cusp in the above theorem.
\end{remark}
Borcherds conjectured that each conjugacy class of the automorphism group of the Leech lattice with non-trivial fixed point lattice corresponds to a holomorphic Borcherds product of singular weight. This is proved for classes of square-free level in \cite{Scheithauer-KacMoody} and \cite{Scheithauer-Class}. The general case is treated in \cite{Scheithauer-Moonshine}. We show that our Borcherds products $\Psi_R$ also fit into this picture. We first apply an Atkin-Lehner involution to the lattice $M$.
\begin{prp}
Let $R$ be one of the root systems from Proposition \ref{prop:rootSystems} and let $M=2U\oplus L_{\text{ev}}(-1)$. Let $N$ be the level of $M$. The lattice 
\[
W_N(M)=\sqrt{N}\left(M^\vee\cap \frac{1}{N}M\right)\subset \R\otimes M
\]
can be written as $U\oplus U(N)\oplus \tilde{L}$ for a negative definite lattice $\tilde{L}$. Let $\Lambda$ be the Leech lattice. There exists an element $g$ in Conway's group $\Co_0=\Orth(\Lambda)$ such that $\tilde{L}(-1)$ is isomorphic to $\Lambda_g$, where $\Lambda_g$ is the lattice of all vectors in $\Lambda$ that are fixed by $g$. In the following table we list $g$ and the genus of $\Lambda_g$.
\[
\renewcommand{\arraystretch}{1.2}
\begin{array}{c|c|c|c}
R  & \text{ class of } g & \text{ cycle shape of } g & \text{ genus of } \Lambda_g\\ \hline

A_4                              & 5C & 1^{-1}5^5 & \II_{4,0}(5^{+3}) \\
A_1\oplus B_3                &  -10D & 1^{-2}2^35^210^1 & \II_{4,0}(2_{\II}^{-4}5^{-3}) \\
A_1\oplus C_3               & -8E & 1^{-2}2^34^18^2 & \II_{4,0}(2_3^{-1}4_1^{+1}8_{\II}^{-2})\\
B_2\oplus G_2                & -12I & 1^{-2}2^23^24^112^1 & \II_{4,0}(2_2^{+2}4_{\II}^{-2}3^{+3})\\  \hline
3A_2                              & 3C & 1^{-3}3^9 & \II_{6,0}(3^{+5}) \\                                               
3A_1\oplus A_3                &  -4C & 1^{-4}2^64^4 & \II_{6,0}(2_6^{+2}4_{\II}^{+4})\\    
2A_1\oplus A_2\oplus B_2  &   -6C & 1^{-4}2^53^46 & \II_{6,0}(2_{\II}^{-6}3^{-5})\\\hline                                               
8A_1                                  &  -2A & 1^{-8}2^{16} & \II_{8,0}(2_{\II}^{+8})
\end{array}
\] 
\end{prp}
\begin{proof}
Since $M$ has level $N$, we have that $M^\vee \subset\frac{1}{N}M$, which yields that $W_N(M)\cong M^\vee(N)$.
The statement that the lattice $W_N(M)$ is of the form $U\oplus U(N)\oplus \tilde{L}$ can be checked separately for each $R$. In all of these cases the genus of $U\oplus U(N)\oplus \Lambda_g$ contains only one class, so to finish the proof it suffices to prove that the genera of $\tilde{L}(-1)$ and $\Lambda_g$ coincide, which can be checked for each of the cases separately.
\end{proof}
\begin{remark}
In those cases where the level $N$ of $M$ is square-free, Scheithauer gave the following natural construction of $\Psi_R$. Let $K=U\oplus U(N)\oplus \Lambda_g$ and define $\eta_g$ by the cycle shape of $g$, i.e. if $g$ has characteristic polynomial $\prod (X^b-1)^{r_b}$, we define $\eta_g(\tau)=\prod \eta(b\tau)^{r_b}$. Then the scalar-valued modular form $1/\eta_g$ of weight $-\rk(\Lambda_g)/2$ can be lifted to a vector-valued modular form $F_{\Gamma_0(N),1/\eta_g,0}$ for the Weil representation $\rho_{D(K)}$. The Fourier coefficients of the principal part of this vector-valued modular form are non-negative integers, hence we obtain a holomorphic Borcherds product $\Psi_g$ on $K(-1)$. This Borcherds product $\Psi_g$ has singular weight. Its expansion at a level $N$ cusp is the twisted denominator identity of the fake monster algebra corresponding to $g$. 

The constructions of $\Psi_g$ and $\Psi_R$ are related in the following way. As explained in the previous paragraph, the function $\Psi_g$ is given by $B(F_{\Gamma_0(N),1/\eta_g,0})$. Similarly, by Theorem \ref{thm:ReflectiveProducts}, the function $\Psi_R$ is given by $B(F_{\Gamma_0(N),f,0})$ for a suitable modular form $f$. It turns out that $f$ is up to a constant given by the Atkin-Lehner involution $W_N(1/\eta_g)$. We can therefore say that $\Psi_R$ is obtained from $\Psi_g$ by taking the Atkin-Lehner involution of both the lattice $K(-1)$ and the input function $1/\eta_g$.

There is another way to see the relationship between the two modular forms. Since $K^\vee(N)=U(N)\oplus U\oplus \Lambda_g^\vee(N)$ is isomorphic to $2U\oplus L_\text{ev}$, we have
\[
\Orth^+(K(-1))=\Orth^+(K^\vee(-1))=\Orth^+(K^\vee(-N))\cong \Orth^+(M),
\]
and $\Psi_g$ can be viewed as a modular form for $\Orth^+(M)$. We then obtain $\Psi_g=\Psi_R$ by comparing their divisors or their Fourier expansions at suitable $0$-dimensional cusps.
\end{remark}
\begin{remark}
    In the three cases where the level of $M$ is not square-free, one can still construct a strongly-reflective Borcherds product $\Psi_g$ of singular weight whose expansion at a level $N$ cusp is the twisted denominator identity of the fake monster algebra corresponding to $g$. However, one has to replace the vector-valued modular form $F$ on $D(K)$ by $F=F_{\Gamma_0(N),1/\eta_g,0}+F_{\Gamma_0(N),h,D^{N/2}}$ for a suitable scalar-valued modular form $h$ (see \cite{Scheithauer-Moonshine}). After identifying $\Orth^+(K(-1))$ with $\Orth^+(M)$ we again obtain $\Psi_g=\Psi_R$.
\end{remark}

\section{Proof of the main theorem}\label{section:proof}
By comparing their divisors, we want to prove that the Borcherds product $\Psi_R$ constructed in Theorem \ref{thm:ReflectiveProducts} equals $G(\vartheta_R)$ for all root systems $R$ from Proposition \ref{prop:rootSystems}. In order to do so, we first prove that $G(\vartheta_R)$ is a modular form for the full modular group $\Orth^+(M)$. 
\begin{lem}\label{lem:ModularFullGroup}
    Let $R$ be one of the root systems from Proposition \ref{prop:rootSystems}. Then $G(\vartheta_R)$ is a modular form for the full group $\Orth^+(M)$ (possibly with a character). 
\end{lem}
\begin{proof}
    Let $D=D(L_{\text{ev}})$ and let $\theta_R$ be the modular form of weight $0$ for $\rho_D$ corresponding to $\vartheta_R$ under \eqref{eq:Correspondence}. Then the Gritsenko lift of $\vartheta_R$ is the additive  Borcherds lift $\Phi_R$ of $\theta_R$. The additive Borcherds lift is constructed as an integral of the inner product of $\theta_R$ and the Siegel theta function (see \cite[Section 6]{Bor98}). The Siegel theta function is invariant under $\Orth^+(M)$. This implies that for every automorphism $\sigma\in \Orth^+(M)$, the additive lift of $\sigma(\theta_R)$ equals $\sigma(\Phi_R)$, where the action of $\sigma$ on $\theta_R$ is given by its action on $D$. Therefore, if $\theta_R$ is invariant under $\Orth(D)$ up to a character, then $\Phi_R$ is a modular form for the full group $\Orth^+(M)$ (with character given by the lift of the character of $\theta_R$ to $\Orth^+(M)$). The invariance of $\theta_R$ under $\Orth(D)$ can be checked for each of the root systems $R$. This is done in the following lemma.
\end{proof}
\begin{lem}\label{lem:Weilinvariants}
    Let $R$ be one of the root systems from Proposition \ref{prop:rootSystems} and let $D=D(L_{\text{ev}})$. Let $\theta_R$ be the modular form of weight $0$ for $\rho_D$ corresponding to one of the functions $\vartheta_R$ under \eqref{eq:Correspondence}. Then $\theta_R$ is invariant under $\Orth(D)$ up to a character of order $2$.
\end{lem}
\begin{proof}
    The space of holomorphic modular forms of weight $0$ for the Weil representation $\rho_D$ is the space of invariants of $\rho_D$ and can be computed using \cite[Algorithm 4.2]{ES}. If $R$ is one of $A_4, A_1\oplus C_3, B_2\oplus G_2, 3A_2$ and  $3A_1\oplus A_3$, then the dimension of this space is $1$. It follows that $\theta_R$ is invariant under $\Orth(D)$ up to a character. Since the space of holomorphic modular forms of fixed weight $k$ for $\rho_D$ has a basis consisting of modular forms with integer coefficients (see \cite[Theorem 5.6]{McGraw}), this character must have order at most $2$.
    
    In the other cases, the discriminant form $D$ can be decomposed as $D=D_2\oplus D'$, where $D_2=2_{\II}^{+2}$ and $D'=5^{+3},3^{-3}$ or $1$. The space of invariants of $\rho_D$ is the tensor product of the spaces of invariants of $\rho_{D_2}$ and $\rho_{D'}$. The first space has dimension $2$ and is spanned by $v_1=\textbf{e}_0+\textbf{e}_{\gamma_1}$ and $v_2=\textbf{e}_0+\textbf{e}_{\gamma_2}$, where $\gamma_1$ and $\gamma_2$ are generators of $D_2$ with $q(\gamma_1)=q(\gamma_2)=0\mod 1$. The space of invariants of $D'$ is $1$-dimensional for all of the three cases. Therefore, the space of invariants of $\rho_D$ is two dimensional. Let $v$ be the tensor product of $v_1-v_2$ and a generator of the space of invariants of $D'$. Then $v$ is invariant under the action of $\Orth(D)$ up to a character of order $2$. It therefore suffices to prove that $\theta_R$ is a multiple of $v$.
    
    For all of the three cases, the lattice $L$ is odd. There thus exists a vector $x\in L$ such that $(x,x)$ is odd. The transformation formula for Jacobi forms of lattice index yields $\vartheta_R(\tau,\mathfrak{z}+x)=-\vartheta_R(\tau,\mathfrak{z})$. Since $\vartheta_R$ is given by
    \[
    \vartheta_R(\tau,\mathfrak{z})=\sum_{\gamma\in D}(\theta_R)_{\gamma}\sum_{\ell\in \gamma+L_{\text{ev}}}q^{(\ell,\ell)/2}\zeta^\ell,
    \]
    this condition forces $(\theta_R)_{\gamma}=0$ unless $(\gamma,x)$ is in $1/2+\Z$. Therefore, $\gamma$ can not be an element of $D^2$, which implies that $\theta_R$ is a multiple of $v$.
\end{proof}

Before we can prove Theorem \ref{thm:mainThm}, we also need the following lemma.
\begin{lem}\label{lem:transitiveOnRoots}
    Let $\Psi_R$ be one of the strongly-reflective modular forms of Theorem \ref{thm:ReflectiveProducts} and let $v$ and $v'$ be two primitive vectors of $M^\vee$ such that $(v,v)=(v',v')$ and $v$ and $v'$ have the same order in $D(M)$. If $\Psi_R$ vanishes along both divisors $\cD_{v}$ and $\cD_{v'}$, then $v$ and $v'$ are conjugate under $\Orth^+(M)$.
\end{lem}
\begin{proof}
    First suppose that $M$ has square-free level. This is the case for all $R$ except $R=A_1\oplus C_3, B_2\oplus G_2$ and $3A_1\oplus A_3$. The elements $v$ and $v'$ have the same norm and the same order in $D=D(M)$. By \cite[Proposition 5.1]{Scheithauer-ModForms} and the paragraph after \cite[Proposition 5.2]{Scheithauer-ModForms}, there exists an element $\sigma\in \Orth(D)$ such that $\sigma(v)=v'\mod M$. The projection from $\Orth(M)$ to $\Orth(D)$ is surjective by \cite[Theorem 1.14.2]{Nikulin}. The reflection at a norm $2$ element in one of the hyperbolic planes is an element of $\Orth(M)\setminus \Orth^+(M)$ and has trivial image in $\Orth(D)$. Therefore, the images of $\Orth(M)$ and of $\Orth^+(M)$ in $\Orth(D)$ are the same. We therefore find an element $\sigma'\in \Orth^+(M)$ such that $\sigma'(v)=v'\mod M$. The Eichler criterion (see e.g. \cite[Proposition 3.3]{GHS}) then yields that $v'$ is conjugate to $\sigma'(v)$ and hence also to $v$ under $\Orth^+(M)$.
    
    If the level of $M$ is not square-free, we can use the same argument, except that we cannot apply \cite[Proposition 5.1]{Scheithauer-ModForms} to show that there exists an element $\sigma\in \Orth(D)$ with $\sigma(v)=v'\mod M$. However, it is not difficult to prove this by hand for each of the three remaining cases. 
    
    As an example, we do the case $B_2\oplus G_2$. The lattice $M$ has genus $\II_{6,2}(2_6^{+2}4_{\II}^{-2}3^{-3})$. We can decompose $D=D_4\oplus D_3$. The discriminant form $D_4$ can be decomposed as $D_4=A\oplus B$, where $A\cong 2_6^{+2}$ is generated by elements $\gamma_1$ and $\gamma_2$ of order $2$ with $q(\gamma_1)=q(\gamma_2)=3/4\mod 1$ and $b(\gamma_1,\gamma_2)=0\mod 1$, and $B\cong 4_{\II}^{-2}$ is generated by elements $\delta_1$ and $\delta_2$ of order $4$ with $q(\delta_1)=q(\delta_2)=b(\delta_1,\delta_2)=1/4\mod 1$.  The modular form $\Psi_R$ is the Borcherds product corresponding to the vector-valued modular form 
    \[F=F_{\Gamma_0(12),\eta_{1^{-1}3^{-1}4^{-2}6^{-2}12^2},0}+F_{\Gamma_0(12),\eta_{1^{-4}4^16^{-2}12^1},D^6}.
    \]
    The Fourier expansion of $F$ was given in Theorem \ref{thm:ReflectiveProducts}.  For $a\in \{1/2,1/3,1/4,1/6,1/12\}$ we let $R_a=\{\gamma\in D : F_\gamma=q^{-a}+O(q^{1-a})\}$. We need to prove that $\Orth(D)$ is transitive on $R_a$. 
    
    For $a=1/2$ there is nothing to show because $R_a$ consists of a single element. 
    
    The discriminant form $D_3$ has prime level, so we can apply \cite[Proposition 5.1]{Scheithauer-ModForms} to prove the transitivity of $\Orth(D_3)$ and hence of $\Orth(D)$ on $R_a$ for $a=1/3$ and $a=1/6$. 
    
    That $\Orth(D)$ is transitive on $R_{1/4}$, i.e. on the set of elements $\gamma$ of order $4$ with $q(\gamma)=1/4\mod 1$, can be easily checked by hand. 
    
    Similarly, to prove that $\Orth(D)$ is transitive on $R_{1/12}$, we note that $R_{1/12}$ consists of all elements of the form $\alpha+\beta$ with $\alpha\in D_3$ and $\beta\in D_4$ such that $q(\alpha)=1/3\mod 1$, while $\beta$ has order $4$ with $q(\beta)=3/4\mod 1$. As remarked above, $\Orth(D_3)$ is transitive on the set of all such $\alpha$ and the transitivity of $\Orth(D_4)$ on all such $\beta$ can again be easily checked by hand.
\end{proof}
With the help of the fact that $G(\vartheta_R)$ is modular for the full group $\Orth^+(M)$, we can prove that the divisor of the Borcherds product $\Psi_R$ is contained in the divisor of $G(\vartheta_R)$.

\begin{prp}\label{prp:DivisorsContained}
    For all root systems $R$ from Proposition \ref{prop:rootSystems}, the divisor of $\Psi_R$ is contained in the divisor of $G(\vartheta_R)$.
\end{prp}
\begin{proof}
        Let $N$ be the level of $M$. From Theorem \ref{thm:ReflectiveProducts}, we know that the only possible zeros of $\Psi_R$ are simple zeros along the divisor $\cD_v$ for primitive $v\in M^\vee$ of norm $(v,v)=-2/d$ and order $d$ in $D=D(M)$ for divisors $d>1$ of $N$. Moreover, $\Psi_R$ has a simple zero at such a divisor $\cD_v$ if and only if the image of $v$ in $D$ is contained in the set $R_{1/d}$ defined in the proof of Lemma \ref{lem:transitiveOnRoots}. In view of Lemmas \ref{lem:ModularFullGroup} and \ref{lem:transitiveOnRoots} it suffices to prove that for each divisor $d>1$ of $N$ there is a primitive vector $v\in M^\vee$ of norm $(v,v)=-2/d$ whose image in $D(M)$ is contained in $R_{1/d}$ and such that $G(\vartheta_R)$ vanishes on $\cD_v$.
        
    First suppose $R$ is one of the root systems for which $L_{\text{ev}}$ has square-free level $N$ (as mentioned before, these are all cases, except $R=A_1\oplus C_3, B_2 \oplus G_2$ and $3A_1\oplus A_3$). In these cases, $R_{1/d}$ consists of all elements $\gamma\in D$ with $\ord(\gamma)=d$ and $q(\gamma)=1/d\mod 1$. Using Remark \ref{rmk:Roots}, it is not difficult to see that we can find a root $r\in R^+$ such that $\lambda=r/h\in  L_{\text{ev}}^\vee$ (if $R$ is the direct sum of irreducible root systems $R_i$, then $r\in R_i$ for a unique $i$ and we define $h=h_i$) has order $d$ in $L_{\text{ev}}^\vee/L_{\text{ev}}$ and satisfies $\langle\lambda,\lambda\rangle=2/d$. The function $\vartheta_R(\tau,\mathfrak{z})$ vanishes along the divisor
    \[
    \{(\tau,\mathfrak{z})\in \HH\times (\C\otimes L_{\text{ev}}) : \langle\lambda,\mathfrak{z}\rangle\in \Z\tau+\Z\}.
    \]
    By Remark \ref{rmk:ZerosOfGritsenko}, $G(\vartheta_R)$ then vanishes along the divisor $\cD_{v}$, where $v=(0,0,\lambda,1,0)\in M^\vee$. Note that $v$ is a primitive vector in $M^\vee$ of norm $(v,v)=-2/d$ and order $d$ in $D$. This completes the proof for these cases.
    
    We now prove the statement for the case of $R=B_2\oplus G_2$. For $d=3,4$ and $12$, the image of every primitive $v\in M^\vee$ of norm $(v,v)=-2/d$ and order $d$ in $D$ is in $R_{1/d}$. For these $d$ the proof can be completed as in the case of square-free level. 
    
    We next look at the case $d=6$. We write $L=L_1\oplus L_2$, where $L_1=\Lambda(B_2^\vee)(h_1)$ and $L_2=\Lambda(G_2^\vee)(h_2)$, where $h_1=\frac{1}{2}\sum_{r\in B_2^+}(r,r)$ and $h_2=\frac{1}{2}\sum_{r\in G_2^+}(r,r)$. Since $L_2$ is already even, we have $L_{\text{ev}}={L_1}_{\text{ev}}\oplus L_2$ and $L_{\text{ev}}^\vee/L_{\text{ev}}={L_1}_{\text{ev}}^\vee/{L_1}_{\text{ev}}\oplus L_2^\vee/L_2$. The discriminant form $L_2^\vee/L_2$ is isomorphic to $4_{\II}^{-2}3^{-1}$ and ${L_1}_{\text{ev}}^\vee/{L_1}_{\text{ev}}$ is isomorphic to $2_{6}^{+2}3^{+2}$. Let $r$ be a short root of $B_2$ and $\lambda=r/h_1$. As before, we see that $G(\vartheta_R)$ vanishes along the divisor $\cD_v$ for $v=(0,0,\lambda,1,0)\in M^\vee$. But the image of $v$ in $D$ is equal to the image of $\lambda$, which lies in $2_{6}^{+2}3^{+2}$. From the singular part of $F$ given in Theorem \ref{thm:ReflectiveProducts}, we see that every element $\gamma\in 2_{6}^{+2}3^{+2}$  of order $6$ with $q(\gamma)=1/6\mod 1$ is in $R_{1/6}$. In particular, the image of $v$ in $D$ is in $R_{1/6}$. This completes the proof for $d=6$. 
    
    The case $d=2$ is more complicated. To prove this case we let $r$ be as above, i.e. a short root of $B_2$. Then $\langle r, r\rangle=3$ and $r$ has order $2$ in ${L_1}_{\text{ev}}^\vee/{L_1}_{\text{ev}}$. Let $v=(0,1,r,1,0)\in M^\vee$. Then $(v,v)=-1$ and $v$ and $r$ have the same image in $D$, which is the unique element in $R_{1/2}$. We need to show that $G(\vartheta_R)$ vanishes along $\cD_v$. This is proved in the next lemma.
    
    The arguments for the cases $R=A_1\oplus C_3$ and $R=3A_1\oplus A_3$ are similar to the ones with $d\neq 2$ for the case $R=B_2\oplus G_2$.   
\end{proof}

\begin{lem}
    Let $R=B_2\oplus G_2$ and let $r$ be a short root of $B_2$. Let $v=(0,1,r,1,0)\in M^\vee$ and let $\sigma\in \Orth^+(M)$ be the reflection along $v^\perp$. Then $\sigma(G(\vartheta_R))=-G(\vartheta_R)$. In particular, $G(\vartheta_R)$ vanishes along $\cD_v$.
\end{lem}
\begin{proof}
    Let $D=D(L_{\text{ev}})$ and let $\theta_R$ be the modular form of weight $0$ for $\rho_D$  corresponding to $\vartheta_R$ under \eqref{eq:Correspondence}. Then $G(\vartheta_R)$ is the additive lift of $\theta_R$. Moreover, $\sigma(G(\vartheta_R))$ is equal to the additive lift of $\sigma(\theta_R)$. It therefore suffices to prove that $\sigma(\theta_R)=-\theta_R$, where the action of $\sigma$ on $\theta_R$ is given by its action on $D$. The space of modular forms of weight $0$ for $\rho_D$ is the tensor product of the spaces of modular forms of weight $0$ for $\rho_{D_4}$ and $\rho_{D_3}$, which both have dimension one. Recall that $D_4=A\oplus B$, where $A\cong 2_6^{+2}$ is generated by two elements $\gamma_1$ and $\gamma_2$ of order $2$ with $q(\gamma_1)=q(\gamma_2)=3/4\mod 1$ and $b(\gamma_1,\gamma_2)=0\mod 1$,  and $B\cong 4_{\II}^{-2}$ is generated by elements $\delta_1$ and $\delta_2$ of order $4$ with $q(\delta_1)=q(\delta_2)=b(\delta_1,\delta_2)=1/4\mod 1$. The image of $v$ in $D$ is $\gamma_1+\gamma_2$. It follows that $\sigma$ acts trivially on $D_3$ and on $B$ and it permutes $\gamma_1$ and $\gamma_2$. Using \cite[Algorithm 4.2]{ES}, we can compute a generator $G=\sum_{\gamma\in D_2}G_\gamma \textbf{e}_\gamma$ of the space of modular forms of weight $0$ for $\rho_{D_4}$. We obtain
    \[
    G_\gamma=\begin{cases}
    1 & \text{ if } \gamma\in \{\gamma_1+\delta_1,\gamma_1-\delta_2,\gamma_1-\delta_1+\delta_2,\gamma_2-\delta_1,\gamma_2+\delta_2,\gamma_2+\delta_1-\delta_2\}, \\
    -1 & \text{ if } \gamma\in \{\gamma_1-\delta_1,\gamma_1+\delta_2,\gamma_1+\delta_1-\delta_2,\gamma_2+\delta_1,\gamma_2-\delta_2,\gamma_2-\delta_1+\delta_2\},\\
    0 & \text{ otherwise.}
    \end{cases}
    \] 
    We see that $\sigma(G)=-G$. Since $\theta_R$ is a multiple of the tensor product of $G$ and a modular form of weight $0$ for $\rho_{D_3}$ (which is invariant under $\sigma$), we obtain $\sigma(\theta_R)=-\theta_R$.
\end{proof}
We can now complete the proof of our main theorem.
\begin{proof}[Proof of Theorem \ref{thm:mainThm}]
By Proposition \ref{prp:DivisorsContained}, the divisor of $\Psi_R$ is contained in the divisor of $G(\vartheta_R)$. Therefore, the quotient of $G(\vartheta_R)$ by $\Psi_R$ is a holomorphic modular form of weight $0$ and therefore constant. Comparing the first Fourier-Jacobi coefficient, we see that $G(\vartheta_R)
=B(\varphi)$ for some $\varphi$. By Remark \ref{rmk:ShapeOfBorcherds} the Jacobi form $\varphi$ must be equal to $-\vartheta_R|T_-(2)/\vartheta_R$. This completes the proof.
\end{proof}
\begin{cor}\label{cor:MainCorollary}
 The theta block conjecture is true for the pure theta blocks from Table \ref{table}.
\end{cor}
\begin{proof}
Let $x=(x_f)_{f\in F}\in \underline{R}$ be an integer vector with $x_f\neq 0$ for all $f$. Let $K$ be the lattice $\Z$ with bilinear form $(u,v)=muv$, where $m=(x,x)$. Recall that we defined $s_x\colon K\to \underline{R}$ by $s_x(u)=ux$ and
\[
s_x^*\colon J_{k,\underline{R}}\to J_{k,K}=J_{k,\frac{m}{2}}, \quad \varphi(\tau,\mathfrak{z})\mapsto \varphi(\tau,s_x(z)) \quad (z\in \C\otimes K).
\]

Each of the pure theta blocks from Table \ref{table} is of the form $s_x^*\vartheta_R$ for such a vector $x\in \underline{R}$.
Each of the pure theta blocks of integral index from Table \ref{table} is of the form $s_x^*\vartheta_R$ for some vector $x\in L_\text{ev}$. For the theta block conjecture, we only care about theta blocks of integral index. Let us assume that $x\in L_\text{ev}$ and $s_x^*\vartheta_R$ is not identically zero. 
We also denote by $s_x^*$ the pullback of a modular form $F$ on $\cD(2U\oplus L_\text{ev}(-1))^\bullet$ to $\cD(2U\oplus K(-1))^\bullet$. 
From the defintion of the Gritsenko lift and the linear action of $T_-(m)$ in the variable $\mathfrak{z}$, we see that $G(s_x^*\varphi)=s_x^*G(\varphi)$ for any $\varphi\in J_{k,L_\text{ev}}$. Similarly, we have $B(s_x^*\varphi)=s_x^*B(\varphi)$ for any $\varphi\in J_{0,L_\text{ev}}^!$ with integral singular coefficients, whenever $s_x^*\varphi\neq 0$.  Since $T_-(2)$ also commutes with $s_x^*$, this completes the proof.
\end{proof}

We end this paper with several remarks.

\begin{remark}
Like the cases $R=A_4, 3A_2, 8A_1$, when $R=3A_1\oplus A_3$, the associated lattice $L_\text{ev}=S_6$ also satisfies the following $\operatorname{norm}_2$ condition:
$$
\operatorname{norm}_2:\ \forall\, \bar{c} \in L^\vee/L \quad \exists\, h_c \in \bar{c} 
\quad \text{such that} \quad (h_c,h_c)\leq 2.
$$
Thus we can use the much simpler method in \cite{GW} to prove this case.
\end{remark}

\begin{remark}
It is easy to check directly that each $\vartheta_R$ appears as the first Fourier-Jacobi coefficient of the Borcherds product $\Psi_R$ constructed in Theorem \ref{thm:ReflectiveProducts}. Since $\Psi_R$ is holomorphic, its Fourier-Jacobi coefficients will be holomorphic Jacobi forms. This provides a new proof that $\vartheta_R$ is holomorphic at infinity (i.e. Theorem \ref{thm:RootThetaFunction}) for all root systems from Proposition \ref{prop:rootSystems}.
\end{remark}

\begin{remark}
When $R=A_1\oplus B_3$, $L_\text{ev}=L_4$ (see \S \ref{sec:5-thetablocks}). There are two different embeddings of $L_4$ into $A_4^\vee(5)$. The associated two pull-backs of $\vartheta_{A_4}$ from $A_4^\vee(5)$ to $L_4$ give two theta blocks which are Jacobi forms of weight 2 and index $L_4$. This gives a basis of $J_{2,L_4}$ because we know from the proof of Lemma \ref{lem:Weilinvariants} that $\dim J_{2, L_4}=2$. Their specializations are as follows.
\begin{align*}
\theta^{(1)}_{A_4}&=\eta^{-6}
\vartheta_{a}\vartheta_{b}\vartheta_{b+c}\vartheta_{b+2c+2d}
\vartheta_{a+b}\vartheta_{b+c+2d}\vartheta_{c}
\vartheta_{a-c}\vartheta_{c+2d}\vartheta_{a+b+c+2d},\\
\theta^{(2)}_{A_4}&=\eta^{-6}
\vartheta_{a-c-d}\vartheta_{b}\vartheta_{b+c}\vartheta_{b+2c+2d}
\vartheta_{a+b+c+d}\vartheta_{b+c+2d}\vartheta_{c}
\vartheta_{a+d}\vartheta_{c+2d}\vartheta_{a+b+d}.
\end{align*}
The same specialization of $\vartheta_{A_1\oplus B_3}$ gives
$$
\theta_{A_1\oplus B_3}=\eta^{-6}\vartheta_{2a+b+d}\vartheta_{b}\vartheta_{b+c}\vartheta_{b+2c+2d}
    \vartheta_{b+c+d}\vartheta_{b+c+2d}\vartheta_{c}
    \vartheta_{c+d}\vartheta_{c+2d}\vartheta_{d}
$$
and we have the two identities
\begin{align*}
\theta_{A_1\oplus B_3}&=\theta^{(1)}_{A_4}-\theta^{(2)}_{A_4},\\
B\left(-\frac{\theta_{A_1\oplus B_3}|T_{-}(2)}{\theta_{A_1\oplus B_3}}
\right)&=B\left(-\frac{\theta^{(1)}_{A_4}|T_{-}(2)}{\theta^{(1)}_{A_4}}
\right)-B\left(-\frac{\theta^{(2)}_{A_4}|T_{-}(2)}{\theta^{(2)}_{A_4}}
\right).
\end{align*}

We get similar results when we embed $L_6$ into $3A_2$.
\end{remark}

\begin{remark}
As an application, we can construct special orthogonal modular forms using our reflective Borcherds products $\Psi_R$ of singular weight. We discuss an interesting example in the case $R=2A_1\oplus A_2\oplus B_2$. In this case, the lattice $L_\text{ev}=L_6$ can be decomposed as a direct sum of $A_2$ and a lattice $T_4$ of rank $4$. The quasi pull-back of $\Psi_R$ from $\cD(2U\oplus L_6(-1))$ to $\cD(2U\oplus T_4(-1))$ gives a strongly-reflective cusp form of canonical weight $6$ (see \cite{Gritsenko-Reflective} for the details of quasi pull-backs). By \cite[Theorem 1.5]{Gritsenko-Reflective}, the corresponding modular variety has geometric genus 1 and Kodaira dimension 0.
\end{remark}

\bigskip

\noindent
\textbf{Acknowledgements} 
We thank Nils Scheithauer for helpful discussions on the content of this paper and for providing us his unpublished notes \cite{Scheithauer-Moonshine}. 
M. Dittmann acknowledges support by the LOEWE research unit "Uniformized
Structures in Algebra and Geometry" and by Deutsche
Forschungsgemeinschaft (DFG) through the Collaborative Research Centre
TRR 326 "Geometry and Arithmetic of Uniformized Structures", project
number 444845124.
H. Wang would like to thank Valery Gritsenko, Nils-Peter Skoruppa and Brandon Williams for many helpful discussions, and he is grateful to Max Planck Institute for Mathematics in Bonn for its hospitality and financial support where this work was done.  H. Wang was supported by the Institute for Basic Science (IBS-R003-D1). 

\vspace{3mm}



\textbf{Data availability} This  paper has no associated data.

\bibliographystyle{amsalpha}
\bibliography{bibtex}

\providecommand{\bysame}{\leavevmode\hbox to3em{\hrulefill}\thinspace}
\providecommand{\MR}{\relax\ifhmode\unskip\space\fi MR }
\providecommand{\MRhref}[2]{%
  \href{http://www.ams.org/mathscinet-getitem?mr=#1}{#2}
}
\providecommand{\href}[2]{#2}
\begin{thebibliography}{GPY15}

\bibitem[Bor95]{Bor95}
Richard~E. Borcherds, \emph{{Automorphic forms on $O_{s+2,2}(R)$ and infinite
  products}}, Invent. Math. \textbf{120} (1995), 161 -- 213.

\bibitem[Bor98]{Bor98}
\bysame, \emph{{Automorphic forms with singularities on Grassmannians}},
  Invent. Math. \textbf{132} (1998), no.~3, 491--562.

\bibitem[CS99]{CS99}
J.~H. Conway and N.~J.~A. Sloane, \emph{Sphere packings, lattices and groups},
  third ed., Grundlehren der mathematischen Wissenschaften [Fundamental
  Principles of Mathematical Sciences], vol. 290, Springer-Verlag, New York,
  1999, With additional contributions by E. Bannai, R. E. Borcherds, J. Leech,
  S. P. Norton, A. M. Odlyzko, R. A. Parker, L. Queen and B. B. Venkov.

\bibitem[ES17]{ES}
Stephan Ehlen and Nils-Peter Skoruppa, \emph{{Computing Invariants of the Weil
  Representation}}, L-Functions and Automorphic Forms (Jan~Hendrik Bruinier and
  Winfried Kohnen, eds.), 2017, pp.~81 -- 96.

\bibitem[EZ85]{EZ85}
Martin Eichler and Don Zagier, \emph{{The theory of Jacobi forms}}, Progress in
  Mathematics, 55. Birkhäuser Boston, Inc., Boston, MA, 1985.

\bibitem[GHS09]{GHS}
Valery Gritsenko, Klaus Hulek, and Gregory~K. Sankaran, \emph{{Abelianisation
  of orthogonal groups and the fundamental group of modular varieties}}, J.
  Algebra \textbf{322} (2009), no.~2, 463 -- 478.

\bibitem[GPY15]{GPY}
Valery Gritsenko, Cris Poor, and David~S. Yuen, \emph{{Borcherds Products
  Everywhere}}, J. Number Theory \textbf{148} (2015), 164 -- 195.

\bibitem[Gri94]{Gri94}
Valery Gritsenko, \emph{{Modular forms and moduli spaces of Abelian and
  $\Kthree$ surfaces}}, Algebra i Analiz \textbf{6} (1994), no.~6, 65 -- 102,
  English translation in St. Petersburg Math. J. {\bf 6} (1995), no. 6, 1179
  -1208.

\bibitem[Gri18]{Gritsenko-Reflective}
\bysame, \emph{{Reflective modular forms and applications}}, Russian Math.
  Surveys \textbf{73} (2018), no.~5, 797 -- 864.

\bibitem[GSZ19]{GSZ}
Valery Gritsenko, Nils-Peter Skoruppa, and Don Zagier, \emph{{Theta Blocks}},
  2019, arXiv:1907.00188.

\bibitem[GW20]{GW}
Valery Gritsenko and Haowu Wang, \emph{{Theta block conjecture for paramodular
  forms of weight 2}}, Proc. Amer. Math. Soc. \textbf{148} (2020), 1863 --
  1878.

\bibitem[Hum72]{Humphreys}
James~E. Humphreys, \emph{{Introduction to Lie Algebras and Representation
  Theory}}, first ed., Springer-Verlag New York, 1972.

\bibitem[McG03]{McGraw}
William~J. McGraw, \emph{{The rationality of vector valued modular forms
  associated with the Weil representation}}, Math. Ann. \textbf{326} (2003),
  105--122.

\bibitem[Nik80]{Nikulin}
Viacheslav~V. Nikulin, \emph{{Integral symmetric bilinear forms and some of
  their applications}}, Mathematics of the USSR-Izvestiya \textbf{14} (1980),
  no.~1, 103--167.

\bibitem[Sch]{Scheithauer-Moonshine}
Nils~R. Scheithauer, \emph{{Moonshine for Conway's group}}, in preparation.

\bibitem[Sch04]{Scheithauer-KacMoody}
\bysame, \emph{{Generalized Kac–Moody algebras, automorphic forms and
  Conway's group I}}, Adv. Math. \textbf{183} (2004), no.~2, 240 -- 270.

\bibitem[Sch06]{Scheithauer-Class}
\bysame, \emph{{On the classification of automorphic products and generalized
  Kac-Moody algebras}}, Invent. Math. \textbf{164} (2006), no.~3, 641--678.

\bibitem[Sch09]{ScheithauerWeil}
\bysame, \emph{{The Weil representation of $\SL_2(\Z)$ and some of its
  applications}}, Int. Math. Res. Not. IMRN \textbf{8} (2009), 1488--1545.

\bibitem[Sch15]{Scheithauer-ModForms}
\bysame, \emph{{Some constructions of modular forms for the Weil representation
  of $\SL_2(\Z)$}}, Nagoya Mathematical Journal \textbf{220} (2015), 1--43.

\end{thebibliography}

\end{document}